\documentclass[11pt]{amsart}
\usepackage{multicol}
\usepackage{amssymb,amsmath,amsthm,bm}
\usepackage[colorlinks=true, urlcolor=blue,linkcolor=blue, citecolor=blue]{hyperref}
\usepackage{mathrsfs,verbatim}
\usepackage{caption,cleveref }
\usepackage{graphicx, float}
\usepackage{color, mathtools, tikz, xcolor}
\usepackage{enumerate,enumitem}
\usepackage{tikz}
\usetikzlibrary{shapes.geometric}
\usepackage{pdfpages}
\usepackage{bbm}
\usepackage{subcaption}

\usetikzlibrary{calc}

\usepackage[mathscr]{euscript}
\usepackage{appendix}
\usepackage{geometry}

\usepackage{fullpage}
\usepackage[justification=centering]{caption}
\numberwithin{equation}{section}

\newtheorem{theorem}{Theorem}[section]

\newtheorem{corollary}[theorem]{Corollary}
\newtheorem{conjecture}[theorem]{Conjecture}
\newtheorem{claim}[theorem]{Claim}

\newtheorem{lemma}[theorem]{Lemma}

\newtheorem{construction}[theorem]{Construction}
	
\newenvironment{proofclaim}[1][Proof of claim]{\begin{proof}[#1]}{\end{proof}}

\setlist{nolistsep}

\title{An oriented discrepancy version of Dirac's theorem}
\author{Andrea Freschi and Allan Lo}
\thanks{AF: University of Birmingham, United Kingdom, {\tt axf079@bham.ac.uk}}
\thanks{AL: University of Birmingham, United Kingdom, {\tt s.a.lo@bham.ac.uk}.
The research leading to these results was supported by EPSRC, grant no. EP/V002279/1 and EP/V048287/1.
There are no additional data beyond that contained within the main manuscript.}


\begin{document}
\maketitle

\begin{abstract}
The study of graph discrepancy problems, initiated by Erd\H os in the 1960s, has received renewed attention in recent years. In general, given a $2$-edge-coloured graph $G$, one is interested in embedding a copy of a graph $H$ in $G$ with large discrepancy (i.e. the copy of $H$ contains significantly more than half of its edges in one colour).

Motivated by this line of research, Gishboliner, Krivelevich and Michaeli considered an oriented version of graph discrepancy for Hamilton cycles. In particular, they conjectured the following generalization of Dirac's theorem: if $G$ is an oriented graph on $n\geq3$ vertices with $\delta(G)\geq n/2$, then $G$ contains a Hamilton cycle with at least $\delta(G)$ edges pointing forwards. In this paper, we present a full resolution to this conjecture. 
\end{abstract}

\section{Introduction}

The study of embedding problems for graphs is a central research topic in extremal combinatorics. Broadly speaking, embedding problems arise as specific instances of the following more general question: given a graph $H$, which sufficient conditions ensure that a host graph $G$ contains a copy of $H$? For example, the celebrated Dirac's theorem~\cite{Dirac} states that a graph $G$ on $n\geq3$ vertices with minimum degree $\delta(G)\geq n/2$ must contain a Hamilton cycle. In this paper, we are interested in an embedding problem concerning the notion of {\it graph~discrepancy}. 

\smallskip

Discrepancy theory is a widely studied subject that encompass various fields in mathematics. Classical discrepancy theory addresses the general problem of partitioning a set $\mathcal U$ in such a way that the elements of $\mathcal U$ are distributed as evenly as possible with respect to some family of subsets~$\mathcal S\subseteq 2^{\mathcal{U}}$. See e.g. the monograph of Beck and Chen~\cite{BC} for more details. 

In graph theory, the concept of {\it graph discrepancy} was introduced to quantify the colour ``imbalance" in a $2$-edge-coloured graph $H$. Informally, we say that a $2$-edge-coloured graph $H$ has {\it large discrepancy} if it has significantly more than half of its edges in one colour. Given this definition, it is natural to pose the following question: which sufficient conditions ensure that a $2$-edge-coloured graph $G$ contains a copy of a graph $H$ with large discrepancy? We will refer to problems of this type as {\it graph discrepancy problems}.

The study of graph discrepancy problems was first raised by Erd\H os in the 1960s (see e.g.~\cite{E,ES}). In 1995, Erd\H os, F\"uredi, Loebl and S\'os~\cite{EFLS} proved the following discrepancy result about trees: there exists a constant $c>0$ such that for every $n$-vertex tree $T$, a 2-edge-coloured complete graph $K_n$ must contain a copy of $T$ with at least $\frac{n-1}{2}+c(n-1-\Delta(T))+O(1)$ edges in one colour. 

In recent years, the study of graph discrepancy problems has received renewed attention, and several works have been produced on the subject (see e.g.~\cite{BCJP,BCPT,FHT,GKM-1,GKM-2}). In particular, Balogh, Csaba, Jing and Pluh\'ar~\cite{BCJP} proved the following graph discrepancy version of Dirac's theorem.

\begin{theorem}[\cite{BCJP}]\label{thm:diracdisc}
Let $0<c<1/4$ and $n\in\mathbb N$ be sufficiently large. If $G$ is a $2$-edge-coloured graph on $n$ vertices with $\delta(G)\geq\left(\frac{3}{4}+c\right)n$, then $G$ contains a Hamilton cycle with at least $(\frac{1}{2}+\frac{c}{64})n$ edges in one colour.
\end{theorem}

In other words, Theorem~\ref{thm:diracdisc} states that if the minimum degree $\delta(G)$ exceeds $\frac{3n}{4}$ then $G$ must contain a Hamilton cycle with more than $\frac{n}{2}$ edges in one colour i.e. with large discrepancy. As $\delta(G)-\frac{3n}{4}$ grows larger, we can find a Hamilton cycle with larger discrepancy too.
The bound of $\delta(G) > \frac{3n}{4}$ is tight as there exists a $2$-edge-coloured graph~$G$ with $\delta(G)=3n/4$ such that all Hamilton cycles in $G$ have $n/2$ edges in each colour. Theorem~\ref{thm:diracdisc} has also been generalised to many colours, see~\cite{FHT,GKM-2}.
\medskip

Motivated by this line of research, Gishboliner, Krivelevich and Michaeli~\cite{GKM} recently considered an alternative notion of graph discrepancy for Hamilton cycles in oriented graphs. Given an oriented graph $G$, for every $x,y\in V(G)$ we write $xy$ to denote the edge oriented from $x$ to $y$. Observe that Hamilton cycles are a convenient structure to consider in the oriented setting, as they have a natural notion of direction. In fact, say $C = v_1 \dots v_{\ell}v_1$ is a cycle in an oriented graph $G$. 
Let $\sigma^+(C)$ and $\sigma^-(C)$ be the number of forward and backward edges in~$C$, respectively.
Formally, let $\sigma^+(C):=| \{1\le i\le\ell: v_{i}v_{i+1} \in E(G)\} |$ and $\sigma^-(C) := |\{1\le i\le\ell: v_{i+1}v_{i} \in E(G)\}|$ (here we take the indices modulo~$\ell$).
Let $\sigma_{\max}(C):=\max\{\sigma^+(C),\sigma^-(C)\}$ and $\sigma_{\min}(C):=\min\{\sigma^+(C),\sigma^-(C)\}$. Observe that $\sigma^+(C)$ and $\sigma^-(C)$ might depend on the labelling of the vertices of $C$, while $\sigma_{\max}(C)$ and $\sigma_{\min}(C)$ depend exclusively on the cycle~$C$.   In general, we say that a Hamilton cycle $C$ in an oriented graph on $n$ vertices has {\it large oriented discrepancy} if $\sigma_{\max}(C)$ is significantly larger than~$n/2$. 

Given this notion, it is natural to seek conditions that force Hamilton cycles of large oriented discrepancy in oriented graphs. In particular, Gishboliner, Krivelevich and Michaeli~\cite{GKM} proposed the following generalization of Dirac's theorem. 

\begin{conjecture}[\cite{GKM}]\label{conj:main}
Let $G$ be an oriented graph on $n\geq3$ vertices. If $\delta(G)\geq \frac{n}{2}$ then there exists a Hamilton cycle~$C$ in $G$ such that $\sigma_{\max}( C )\geq\delta(G)$.
\end{conjecture}

The following construction from~\cite{GKM} shows that  Conjecture~\ref{conj:main} is sharp, in the sense that we cannot guarantee that $G$ contains a Hamilton cycle $C$ with $\sigma_{\max}( C )>\delta(G)$.

\begin{construction}\label{construction}
Given $n,d\in\mathbb N$ with $n>d$, let $G_{n,d}$ be an $n$-vertex graph with vertex set $V(G)=A\cup B$, where $|A|=d$ and $|B|=n-d$, and edge set consisting of all edges incident to $A$. Assign an orientation to the edges of $G_{n,d}$ with the only restriction being that edges incident to $B$ must be oriented from $B$ to~$A$.
\end{construction}

Indeed, let $n,d\in\mathbb N$ with $n>d\geq n/2$. Clearly, $\delta(G_{n,d})=d$. Let $C=v_1v_2\dots v_n v_1$ be a Hamilton cycle in~$G_{n,d}$. Observe that every vertex in~$B$ is incident to two edges of~$C$ which are oriented in opposite directions. Since $B$ is an independent set, it follows that $\sigma_{\min}(C)\geq |B| = n-d$. As $\sigma_{\min}(C)+\sigma_{\max}(C)=n$, it follows further that $\sigma_{\max}( C ) \leq d=\delta(G_{n,d})$. 

Notice that Conjecture~\ref{conj:main} does hold for the extremal cases $\delta(G)\in\{\frac{n}{2},\frac{n+1}{2},n-1\}$. In fact, the cases $\delta(G)\in\{\frac{n}{2},\frac{n+1}{2}\}$ are an easy consequence of Dirac's theorem. 
Similarly, it is a well-known fact that an oriented complete graph (i.e. a tournament) contains a Hamilton path with all edges pointing in the same direction~\cite{Redei}. This immediately implies Conjecture~\ref{conj:main} holds when $\delta(G)=n-1$. 

Additionally, Gishboliner, Krivelevich and Michaeli proved the following approximate version of Conjecture~\ref{conj:main}.

\begin{theorem}\cite{GKM}\label{thm:old}
Let $k\geq0$ and let $G$ be an oriented graph on $n\geq 30+4(k-1)$ vertices. If $\delta(G)\geq\frac{n+8k}{2}$ then there exists a Hamilton cycle~$C$ in $G$ with $\sigma_{\max}( C )\geq\frac{n+k}{2}$.
\end{theorem}

Observe that the minimum degree conditions in Theorem~\ref{thm:old} and Conjecture~\ref{conj:main} nearly match when $k$ is small and diverge significantly as $k$ grows. 

Our main result is a full resolution of Conjecture~\ref{conj:main}.

\begin{theorem}\label{thm:main}
Let $G$ be an oriented graph on $n \ge 3$ vertices with $\delta(G) \ge n/2$. 
Then there exists a Hamilton cycle~$C$ in~$G$ such that $\sigma_{\max}( C )\geq \delta(G)$.
\end{theorem}

Similarly, for a path~$P=v_1v_2\dots v_\ell$ in an oriented graph~$G$, we can define $\sigma^+(P)$ and $\sigma^-(P)$ to be the number of edges pointing forwards and backwards respectively.\footnote{Formally, $\sigma^+(P):=| \{1\le i\le\ell-1: v_{i}v_{i+1} \in E(G)\} |$ and $\sigma^-(P) := |\{1\le i\le\ell-1: v_{i+1}v_{i} \in E(G)\}|$.} We let $\sigma_{\max}(P):=\max\{\sigma^+(P),\sigma^-(P)\}$ and $\sigma_{\min}(P):=\min\{\sigma^+(P),\sigma^-(P)\}$.
We obtain the following corollary from Theorem~\ref{thm:main}.

\begin{corollary}\label{cor:paths}
Let $G$ be an oriented graph on $n$ vertices. 
Then there exists a path~$P$ in~$G$ such that $\sigma_{\max}(P) \geq \delta(G)$. 
Furthermore, if $\delta(G)\geq n/2$, then $P$ is Hamiltonian.
\end{corollary}

Indeed, if $\delta(G)\ge\frac{n}{2}$ then Corollary~\ref{cor:paths} follows immediately from Theorem~\ref{thm:main}. 
If $\delta(G)<\frac{n}{2}$, then by Ore's theorem~\cite{Ore} there exists a path $P$ in $G$ with $|P|\ge\min\{n,2\delta(G)\}=2\delta(G)$ and so $\sigma_{\max}(P)\ge\delta(G)$, as required.
Notice that Corollary~\ref{cor:paths} is sharp by considering the oriented graph~$G_{n,d}$ from Construction~\ref{construction}.

\subsection{Notation}

For the rest of the paper, we write $\sigma(\cdot)$ instead of $\sigma_{\min}(\cdot)$, for both cycles and paths.
For $n \in \mathbb{N}$, we write $[n]$ for the set~$\{1, \dots,n\}$. Given a statement $\mathcal P$, the indicator function~$\mathbbm 1(\mathcal P)$ is defined as 
$$\mathbbm 1(\mathcal P):=\begin{cases}
1& \text{ if $\mathcal P$ holds;}\\
0& \text{ otherwise.}
\end{cases}$$
Let~$G$ be an oriented graph. Given a vertex $v  \in V(G)$, we denote the set of neighbours of~$v$ in~$G$ as~$N_G(v)$, or $N(v)$ for brevity. Given a subgraph~$H$ of~$G$, we write $d(v,H) := | N_G(v) \cap V(H)|$ i.e. $d(v,H)$ is the number of neighbours of~$v$ in~$G$ lying in~$V(H)$. We may write $|G|:=|V(G)|$ and $e(G)$ for the number of edges in~$G$. 
Given a set $X\subseteq V(G)$, $G[X]$ is the spanning subgraph of~$G$ with vertex set~$X$. Recall that for every $x,y\in V(G)$ we write $xy$ to denote the edge oriented from~$x$ to~$y$. We say that a path $v_1v_2\dots v_\ell$ in $G$ is {\it directed} if $v_kv_{k+1}\in E(G)$ for every $k\in[\ell-1]$ i.e. $\sigma^-(v_1v_2\dots v_\ell)=0$.

	
\section{Proof of Theorem~\ref{thm:main}}\label{sec:cycles}

Our proof of Theorem~\ref{thm:main} is actually based upon the proof that every tournament has a directed Hamilton path. 
We now give a sketch proof of a special case of Theorem~\ref{thm:main}, which forms the basis of our proof. 
Let $\delta(G) = n - \ell$ with $\ell\le n/2-1$.
To prove Theorem~\ref{thm:main}, it suffices to show that $G$ contains a Hamilton cycle~$C^*$ such that $\sigma (C^*)  \le \ell$ as $\sigma_{\max}(C^*)=|C^*|-\sigma(C^*)=n- \sigma(C^*)$.
We fix $\ell$ and proceed by induction on~$n$. Pick $w\in V(G)$ and a cycle $C = v_1\dots v_{n-1}v_1$ with vertex set $V(C)=V(G)\setminus\{w\}$ which minimises $\sigma^-(C)$. By our inductive hypothesis, we may assume that $\sigma^-(C) \le \ell$. 
We consider the case where $\sigma^-(C)  = \ell$, the backward edges $v_{a_1+1}v_{a_1}, v_{a_2+1}v_{a_2},\dots, v_{a_\ell+1}v_{a_\ell}$ in~$C$ are vertex-disjoint, say $1=a_1<a_1+1<a_2<a_2+1<\dots<a_\ell<a_\ell+1\le n$, and $V(C) \setminus N(w) = \{ v_i : i \in \{a_2, a_3, \dots, a_{\ell}\} \} $. 
Whenever $w$ is adjacent to both $v_i$ and $v_{i+1}$ (where $v_n = v_1$), let $C_i$ be the Hamilton cycle in $G$ obtained by inserting $w$ in between~$v_i$ and~$v_{i+1}$, that is, $C_i = v_1 \dots v_i w v_{i+1} \dots v_{n-1}v_1$. 
Recall that $\sigma^-(C) = \ell$ and $v_2v_1$ is a backward edge in~$C$. 
We must have $wv_1,v_2w \in E(G)$  or else $\sigma(C_1) \le \ell$. 
Our aim is to show that $v w \in E(G)$ for all $v \in N(w) \setminus \{v_1\}$.
For $i \in [2,a_2-2]$, given $v_iw\in E(G)$ we deduce that $v_{i+1}w \in E(G)$ or else $\sigma(C_i) = \ell$. 
Note that $v_1 \dots v_{a_2-1} w v_{a_2+1} \dots v_{n-1} v_1$ is a cycle on $n-1$ vertices in~$G$.  
By the minimality of~$\sigma^-(C)$, $v_{a_2-1}w\in E(G)$ implies $v_{a_2+1} w \in E(G)$. By iterating this argument, which allows us to `jump' pass any non-neighbour $v_{a_i}$ of~$w$, we deduce that $v_{n-1} w \in E(G)$.
Recall that $wv_1 \in E(G)$. 
Then $C' = v_1 \dots v_{n-1} w v_1$ is a Hamilton cycle in~$G$ with $\sigma^-(C') \le \sigma^-(C) = \ell$ as required.

The proof of Theorem~\ref{thm:main} involves in partitioning~$C$ into paths $P_1, \dots, P_q$ where each~$P_i$ consists of an initial set of edges pointing backwards followed by a second set of edges pointing forwards. 
If~$|V(P_i)\cap N(w)|$ is large enough, then we are able to deduce some information about the orientation of the edges between~$V(P_i)$ and~$w$ using a more refined version of the argument above. This is achieved through various intermediate steps, by first considering paths with almost all edges pointing backwards (Claim~\ref{claim:pattern}), then a single path $P_i$ (Claim~\ref{clm:m_j}) and finally consecutive paths~$P_i$ containing many elements of $N(w)$ (Claim~\ref{clm:q^*}). As the $P_i$'s form a partition $C$, a simple averaging argument guarantees the existence of consecutive paths $P_i$ containing many neighbours of $w$ (Claim~\ref{clm:average}).

We will need the following lemma.
The lemma implies Theorem~\ref{thm:main} provided that $G$ can be partitioned into a cycle~$C$ and a path~$P$ with small $\sigma(C)$ and $ \sigma(P)$. 
The proof of the lemma involves constructing a Hamilton cycle by ``inserting~$P$ into~$C$'' in a suitable way. 

\begin{lemma} \label{lem:path+cycle}
Let $G$ be an oriented graph on $n$ vertices with $\delta(G) \ge \frac{n}{2}+1$. Let $P$ be a path in $G$ and $C$ a cycle in $G$ such that $P$ and $C$ are vertex-disjoint, $V(G)=V(P)\cup V(C)$, $\sigma(P) \le 1$ and $2 \le |P| < \delta(G)$. If $\sigma(C)\le\ell - |P|$ for some $\ell\in\mathbb N$, then $G$ contains a Hamilton cycle~$C^*$ such that $\sigma (C^*) \le \ell$.
\end{lemma}

\begin{proof}
Let $P = u_1 \dots  u_s$ and $C = v_1 \dots v_{n-s}v_1$ be as in the statement of the lemma.
The indices of~$v_i$ are considered modulo~$n-s$.  
We may assume that $\sigma^-(P)=\sigma(P)$ and $\sigma^-(C) = \sigma(C)$.
In particular, we have $\sigma^-(P)\leq1$ and $\sigma^-(C)\leq\ell-|P|$. 
Furthermore, we have $2\le s<\delta(G)$. We proceed by induction on~$s$.

First, suppose that $s =2$.
We have $\sigma^-(P) = \sigma(P) = 0$. 
Since 
\begin{align*}
d(u_1,C)+d(u_2,C)\ge 2(\delta(G) - 1) \ge n > |C|,
\end{align*}
there exists $i \in [n-2]$ such that $v_i \in N(u_1)$ and $v_{i+1} \in N(u_2)$. 
Then $C^* = v_1 \dots v_i u_1 u_2 v_{i+1} \dots v_{n-2} v_1$ is a Hamilton cycle in~$G$ such that $\sigma^-( C^* ) \le \sigma^-(C) + 2 \le \ell$. Hence, $\sigma(C^*)\leq\ell$.

Second, suppose that $s =3$. 
Since
$$d(u_1,C)+d(u_3,C)\ge2(\delta(G) - 2) \ge n-2 > |C|,$$
there exists $i \in [n-3]$ such that $v_i \in N(u_1)$ and $v_{i+1} \in N(u_3)$.
Then $C^* = v_1 \dots v_i u_1 u_2u_3 v_{i+1} \dots v_{n-3} v_1$ is a Hamilton cycle in~$G$ such that $\sigma^-( C^* ) \le \sigma^- (C) + 2 + \sigma^-(P) \le \ell$. Hence, $\sigma(C^*)\leq\ell$.

Next, we consider the special case when $s=4$ and $\sigma^-(P)=1$. 

\begin{claim} \label{clm:badpath}
If $s = 4$ and $\sigma^-(P) = 1$, then there exists a Hamilton cycle~$C^*$ in~$G$ with $\sigma(C^*)\le\ell$ or there exists a path $P'$ with $V(P')=V(P)$ and $\sigma(P')=0$. 
\end{claim}

\begin{proofclaim}
Suppose that $s=4$ and $\sigma^-(P) = 1$. 
If $G[V(P)]$ is a tournament, then there exists a  directed path $P'$ with $V(P')=V(P)$. In particular, $\sigma(P')=0$. 

If $G[V(P)]$ is not a tournament, we must have $d(u_1,P) + d(u_4, P) \le 5$, implying that 
\begin{align*}	
	d(u_1,C) + d(u_4, C) \ge 2 \delta(G) - (d(u_1,P) + d(u_4, P)) \ge n-3 > |C|. 
\end{align*}
Thus there exists $i \in [n-4]$ such that $v_i \in N(u_1)$ and $v_{i+1} \in N(u_4)$. 
In particular, the cycle $C^* = v_1 \dots v_i u_1\dots u_4 v_{i+1} \dots v_{n-4} v_1$ in $G$ is Hamiltonian and $\sigma^- ( C^* ) \le \sigma^- (C) + 2 + \sigma^-(P) \le \ell-1$. Hence $\sigma(C^*)\leq\ell$, as required.
\end{proofclaim}

Since the lemma holds for $s\in\{2,3\}$, we may assume that $s\geq4$. By Claim~\ref{clm:badpath}, we may further assume that if $s=4$ then $\sigma^-(P)=0$. In particular, we have 
\begin{align}
    s  \ge 4 +\sigma^-(P).  \label{eqn:s+sigma(P)}
\end{align}

Let $I := \{ i + k  \pmod{|C|}: v_{i} \in N(u_1), k \in [s-2]\}$ and $J := \{j  \pmod{|C|} : v_j \in N(u_s )\}$.
\begin{claim}\label{clm:I+J}
We have $I\cap J\ne\emptyset$.
\end{claim}
\begin{proofclaim}
First, observe that $|J| = d(u_s,C) \ge \delta(G) - (s-1) > 1$ where the last inequality follows from the assumption $s=|P|<\delta(G)$. 
If $|I| = |C|$, then $I \cap J = J \ne \emptyset$.
We may assume that $|C| > |I| \ge  d(u_1,C) + s-3$. 
Finally, 
\begin{align*}
	|I \cap J| & \ge |I| +|J| - |C|  \ge  d(u_1, C) +s-3 + d(u_s,C) - (n-s)\\
	&\ge 2 (  \delta(G) - s+1 ) + 2s - 3  -n \ge 1
\end{align*}
as required. 
\end{proofclaim}
By Claim~\ref{clm:I+J}, there exists some $j\in I\cap J$. Without loss of generality, we may assume that $v_{n-s} \in N(u_1)$ and $v_j \in N(u_s)$ for some $j \in [s-2]$, where we take $j$ to be as small as possible. 
If $j=n-s$, then by minimality of $j$ we have $d(u_s,C)=1$. 
This is a contradiction since $d(u_s,C)\ge\delta(G)-s+1>1$.
Thus $j\ne n-s$ and in particular $v_j$ and $v_{n-s}$ are distinct. If $j =1$, then set $C^* = u_1 \dots u_s v_1 \dots v_{n-s}u_1$ and note that $\sigma^-(C^*) \le \sigma^- (C) + 2 + \sigma^-(P) \le \ell$.
Thus we may assume that $j \in [2,s-2]$.

Note that $u_s \notin N(v_{j-1})$ by the minimality of~$j$. 
If $u_k \in N(v_{j-1})$ with $k \in [s-2] \setminus \{2\}$, then the path $P' = u_{1} \dots u_{k-1}$\footnote{When $k=1$, we take $P'$ to be empty.} and the cycle $C' = v_1 \dots v_{j-1} u_k \dots u_s v_j \dots v_{n-s} v_1$ (see Figure~\ref{fig:C*}(a)) satisfy $\sigma^-(P') \le 1$ and
\begin{align*}
 \sigma^- (C') \le \sigma^- (C) + 2 + \sigma^-(u_k\dots u_s)
 \le\ell-s+3\le \ell - |P'|.
\end{align*}
When $k=1$, $C'$ is a Hamilton cycle and so we are done.  
When $k \in [3, s-2]$, we are done by our induction hypothesis on~$C'$ and~$P'$ as $2 \le |P'| < s < \delta(G)$.
If $u_{s-1}u_s \in E(G)$, then $u_{s-1} \notin N(v_{j-1})$ or else we apply our induction hypothesis on the path $P' = u_{1} \dots u_{s-2}$ and the cycle $C' = v_1 \dots v_{j-1} u_{s-1} u_s v_j \dots v_{n-s} v_1$ (note that $\sigma^-(u_{s-1}u_s)=0$ and so $\sigma^-(C')\le\ell-|P'|$).
Hence $d(v_{j-1},P) \le 1 + \mathbbm{1}(u_s u_{s-1} \in E(G))$.

Note that $u_1 \notin N(v_{1})$ by minimality of~$j$ (or else we can relabel~$v_i$ as~$v_{i-1}$ for every $i\in[n-s]$). 
If $u_{k'} \in N(v_{1})$ with $k' \in [s] \setminus \{1,2,s-1\}$, then consider the path $P'' = u_{k'+1} \dots u_{s}$\footnote{When $k'=s$, we take $P''$ to be empty.} and the cycle $C'' = v_1 \dots v_{n-s} u_1 \dots u_{k'} v_1$  (see Figure~\ref{fig:C*}(b)).
By a similar argument as in the preceding paragraph, if $k'\in[s]\setminus\{1,s-1\}$ or $k'=2$ and $u_1u_2 \in E(G)$ then one can use the inductive hypothesis to conclude the proof.
Hence, we may assume $d(v_{1},P) \le 1 + \mathbbm{1}(u_2 u_{1} \in E(G))$.

In summary,   
\begin{align*}
	d(v_1, P) + d(v_{j-1},P) \le 2+ \mathbbm{1}({u_2 u_1 \in E(G)})+ \mathbbm{1}(u_s u_{s-1} \in E(G)) \le 2+ \sigma^-(P) \le 3. 
\end{align*}
Hence 
\begin{align*}
	| N(v_{1}) \cap \{ v_k: v_{k+1} \in N(v_{j-1}) \} | 
		& \ge d( v_{1}, C) + d( v_{j-1}, C) - |C| \\
		& \ge 2 \delta(G) - d(v_1, P) - d(v_{j-1},P) - (n-s) \ge s-1.
\end{align*}
Therefore there exist at least $s-1$ many $k \in [n-s] $ such that $v_k \in N(v_{1})$ and $v_{k+1} \in N(v_{j-1})$. 
Note that $k \ne 1$, so we may pick $k \notin [ s- 2 ] \cup \{n-s\}$. 
Then $C^* = v_1 \dots v_{j-1} v_{k+1} \dots v_{n-s} u_1 \dots u_s v_j \dots v_{k} v_1$ (see Figure~\ref{fig:C*}(c)) is a Hamilton cycle on~$G$ such that 
\begin{align*}
\sigma^- ( C^* ) \le \sigma^- (C) + 4 + \sigma^-(P) \le \ell - s +4 + \sigma^-(P) \le \ell,
\end{align*}
where the last inequality follows from \eqref{eqn:s+sigma(P)}.
\begin{figure}[!tb]
\centering
\begin{subfigure}{.32\textwidth}
    \centering
    \begin{tikzpicture}
    
    \draw (0,0) circle (2cm) node{$C'$};
    \draw[ultra thick] (45:2) arc (45:-300:2);

    \node at (-1.2,2.6) {$P'$};

    \draw (-2,3) -- (2,3);
    \draw[ultra thick] (-2,3) -- (-0.4,3);
    \draw[ultra thick] (0.4,3) -- (2,3);

    \draw[ultra thick] (0.4,3) -- (60:2);
    \draw[ultra thick] (2,3) -- (45:2);

    \fill (-2,3) circle (2pt) node[anchor=south]{$u_1$};
    \fill (-0.4,3) circle (2pt) node[anchor=south east]{$u_{k-1}$};
    \fill (0.4,3) circle (2pt) node[anchor=south]{$u_k$};
    \fill (2,3) circle (2pt) node[anchor=south]{$u_s$};

    \fill (60:2) circle (2pt) node[anchor=north east]{$v_{j-1}$};
    \fill (45:2) circle (2pt) node[anchor=north east]{$v_j$};
  
      \end{tikzpicture}
      \caption*{(a) Path $P'$ and cycle $C'$}
\end{subfigure}
\hfill
\begin{subfigure}{.32\textwidth}
    \centering
    \begin{tikzpicture}
     
    \draw (0,0) circle (2cm) node{$C''$};
    \draw[ultra thick] (120:2) arc (120:-225:2);

    \node at (1.2,2.6) {$P''$};

    \draw (-2,3) -- (2,3);
    \draw[ultra thick] (-2,3) -- (-0.4,3);
    \draw[ultra thick] (0.4,3) -- (2,3);

    \draw[ultra thick] (-0.4,3) -- (120:2);
    \draw[ultra thick] (-2,3) -- (135:2);

    \fill (-2,3) circle (2pt) node[anchor=south]{$u_1$};
    \fill (-0.4,3) circle (2pt) node[anchor=south]{$u_{k'}$};
    \fill (0.4,3) circle (2pt) node[anchor=south west]{$u_{k'+1}$};
    \fill (2,3) circle (2pt) node[anchor=south]{$u_s$};

    \fill (120:2) circle (2pt) node[anchor=north west]{$v_{1}$};
    \fill (135:2) circle (2pt) node[anchor=north west]{$v_{n-s}$};
  
\end{tikzpicture}
\caption*{(b) Path $P''$ and cycle $C''$}
\end{subfigure}
\hfill
\begin{subfigure}{.32\textwidth}
    \centering
    \begin{tikzpicture}

    \draw (0,0) circle (2cm) node[anchor=north east]{$C^*$};
    \draw[ultra thick] (-2,3) -- (2,3) -- (45:2) arc (45:-45:2) -- (120:2) arc (120:60:2) -- (-60:2) arc (-60:-225:2) -- (-2,3);

    \fill (-45:2) circle (2pt) node[anchor= north west]{$v_k$};
    \fill (-60:2) circle (2pt) node[anchor= north west]{$v_{k+1}$};

    \fill (-2,3) circle (2pt) node[anchor=south]{$u_1$};
    \fill (2,3) circle (2pt) node[anchor=south]{$u_s$};

    \fill (60:2) circle (2pt) node[anchor=south]{$v_{j-1}$};
    \fill (45:2) circle (2pt) node[anchor=west]{$v_j$};
    \fill (120:2) circle (2pt) node[anchor=south]{$v_{1}$};
    \fill (135:2) circle (2pt) node[anchor=east]{$v_{n-s}$};
		
\end{tikzpicture}
\caption*{(c) The Hamilton cycle~$C^*$}
\end{subfigure}
\caption{}\label{fig:C*}
\end{figure}   
\end{proof}


We are now ready to prove Theorem~\ref{thm:main}.

\begin{proof}[{\bf Proof of Theorem~\ref{thm:main}}]
Given a cycle~$C$, note that $|C| = \sigma_{\max}(C)+ \sigma(C)$.
Theorem~\ref{thm:main} is equivalent to the following statement. 
Let $n, \ell \in \mathbb{N}$ with $n \ge \max \{ 2\ell, 3 \}$. 
Let $G$ be an oriented graph on $n$ vertices with $\delta(G)  = n- \ell$.
Then $G$ contains a Hamilton cycle~$C^*$ such that $\sigma (C^*)  \le \ell$.

If $\ell =1$, then $G$ is a tournament on $n$ vertices and contains a directed Hamilton path~$P = v_1 \dots v_n$. 
Then the Hamilton cycle $C^* = v_1 \dots v_n v_1$ satisfies $\sigma (C^*) \le 1$. 
 
We may assume that $\ell \ge 2$ and proceed by induction on~$n$. 
If $n \in \{ 2\ell, 2 \ell+1\}$, then Dirac's theorem implies that $G$ contains a Hamilton cycle~$C^*$, with $\sigma(C^*) \le \lfloor n/2 \rfloor \le \ell$. 
Thus we may assume that $n \ge 2 \ell+2$, and so $\delta(G)\ge n/2 + 1$ (this is needed to apply Lemma~\ref{lem:path+cycle}). 
Suppose to the contrary that $G$ does not contain a Hamilton cycle~$C^*$ with $\sigma (C^*) \le \ell$.

Let $C = v_1 \dots v_{n-1}v_1$ be a cycle in $G$ on $n-1$ vertices such that $\sigma(C)$ is minimum. 
The indices of~$v_i$ are considered modulo~$n-1$. 
Without loss of generality, $\sigma(C) =  \sigma^-(C) $. 
Let $\{w \} = V(G) \setminus V(C)$.
Note that $\delta(G \setminus w) \ge n-\ell -1 = |G\setminus w| - \ell$. 
By removing edges if necessary, we may assume equality holds. Furthermore, the assumption $n\ge 2\ell+2$ implies $n-1\ge\max\{2\ell,3\}$.
By induction hypothesis on~$G\setminus w$, we have $\sigma^- (C) \le \ell$ (since $\sigma(C)$ is minimal).

Suppose that $\sigma^-(C) \le \ell-2$.
Since $\delta(G) \ge n/2 +1> |C|/2 $, there exists $i \in [n-1]$ such that $v_{i},v_{i+1} \in N(w)$.
The Hamilton cycle $C^* = v_1 \dots v_{i} w v_{i+1} \dots v_{n-1}v_1$ (which is obtained from~$C$ by inserting~$w$ in between~$v_i$ and~$v_{i+1}$) satisfies $\sigma^- (C^*) \le \sigma^-(C) +2 \le \ell$, a contradiction. 
Therefore, 
\begin{align*}
 \ell -1 \le \sigma^-(C) \le \ell.
\end{align*}

The next claim considers the case when $v_i, v_j \in N(w)$ and all but at most one edge in the path $v_{i} v_{i+1}\dots v_{j}$ have form $v_{k+1}v_k$ (i.e. all but at most one edge contribute to $\sigma^-(C)$). Lemma~\ref{lem:path+cycle} will allow us to determine the orientation of the edges incident to $v_i,w$ and $v_j,w$.

\begin{claim} \label{claim:pattern}
Let $v_i, v_j \in N(w)$ with $v_i\ne v_j$ and $P_{ij} := v_{i} v_{i+1}\dots v_{j}$. The following statements hold: 
\begin{enumerate} [label={\rm(\roman*)}]
	\item if $\sigma^-(P_{ij}) = e (P_{ij})$, then $v_jw,w v_i \in E(G)$; \label{itm:pattern2}
	\item if $\sigma^-(P_{ij}) = e (P_{ij}) -1$ and $v_iw \in E(G)$, then $v_jw \in E(G)$. \label{itm:pattern3}
\end{enumerate}
\end{claim}

\begin{proofclaim} 
Suppose to the contrary that \ref{itm:pattern2} or \ref{itm:pattern3} does not hold. That is, there exist $v_i\not=v_j$ in $N(w)$ such that either $\sigma^-(P_{ij})=e(P_{ij})$ and $\sigma^{-}(v_iwv_j)\le1$ or $\sigma^-(P_{ij})=e(P_{ij})-1$ and $\sigma^{-}(v_iwv_j)=0$. Without loss of generality, say $i =1$. In both cases, we have 
\begin{align*}
	\sigma^-(P_{1j}) - \sigma^- (v_1wv_j) \ge e(P_{1j}) - 1 = j-2. 
\end{align*}
Let $C' = v_1  w v_{j} \dots v_{n-1}v_1$, so 
\begin{align*}
	\sigma^-(C') & =  \sigma^-(C) - \sigma^-(P_{1j}) +\sigma^- (v_1wv_j) \le \sigma^-(C) - j +2.
\end{align*}
If $j = 2$, then $C'$ is a Hamilton cycle in $G$ with $\sigma(C') \le \sigma^-(C) \le \ell$, a contradiction. 
If $j = 3$, then $C'$ is a cycle on $n-1$ vertices with $\sigma(C') \le \sigma (C) -1$ contradicting the minimality of~$\sigma(C)$. 
If $j \ge 4$, then $P = v_{j-1} v_{j-2} \dots v_2$ is a path with $\sigma(P) \le 1$ and 
\begin{align*}
	2\le|P| & = j-2  \le \sigma^-(P_{1j})  \le \sigma^-(C)  \le \ell < n/2 < \delta(G). 
\end{align*}
Note that $V(P)$ and $V(C')$ partition~$V(G)$ and $\sigma(C') \le \ell - |P|$, hence by Lemma~\ref{lem:path+cycle} $G$ contains a Hamilton cycle~$C^*$ such that $\sigma (C^*) \le \ell$, a contradiction. 
\end{proofclaim}

Since $\delta(G)\ge n/2+1$, there exists $i^* \in [n-1]$ such that $v_{i^*},v_{i^*+1} \in N(w)$.
If $\sigma(C)  = \ell-1$, then by considering the Hamilton cycle $v_1 \dots v_{i^*} w v_{i^*+1} \dots v_{n-1}v_1$, we deduce that
\begin{align}
	v_{i^*}v_{i^*+1}, wv_{i^*}, v_{i^*+1}w \in E(G). \label{eqn:i^*} 
\end{align}
Define $J \subseteq [n-1]$ of size $\ell$ such that 
\begin{align*}
	J = \begin{cases}
		\{ i  \in [n-1] : v_{i+1}v_i \in E(G)\} & \text{if $\sigma(C)  = \ell$,}\\
		\{ i  \in [n-1] : v_{i+1}v_i \in E(G)\} \cup \{i^*\} & \text{if $\sigma(C)  = \ell-1$}.
	\end{cases}
\end{align*}
Recall that $\ell < n-1$, so (by rotating the cycle if necessary) we may assume that $1 \in J$ and $n-1 \notin J$. 
Roughly speaking, $J$ consists of all indices~$i$ such that $v_{i+1}v_i$ contributed to~$\sigma(C)$ together with $i^*$ if and only if $\sigma(C) = \ell-1$. 
Note that $J$ can be partitioned in disjoint intervals $J_1, \dots, J_q$ such that 
\begin{align*}
J_j  &= [a_j, a_j + t_j - 1] \text{ for all $j \in [q]$ and }\\
  1= a_1 < a_1+t_1 & < a_2 <a_2+t_2 < a_3 < \dots < a_q < a_{q}+t_q\le n-1.
\end{align*}
In particular,
\begin{align}
 \sum_{j \in [q]} t_j &= |J| =  \ell. \label{eqn:sumt_j}
\end{align}
Note that if $i^* \notin [a_j, a_j + t_j - 1]$, then $v_{a_j + t_j} v_{a_j + t_j-1} \dots v_{a_j}$ is a directed path. Furthermore, $v_{a_j+t_j}v_{a_j+t_j+1}\dots,v_{a_{j+1}}$ is always a directed path, regardless of $i^*$. Hence, the path $v_{a_j}v_{a_j+1}\dots,v_{a_{j+1}}$ in $C$ consists of an initial set of edges pointing backwards (except for, possibly, the edge $v_{i^*}v_{i^*+1}$) and a second set of edges pointing forwards. Our aim is to determine the orientation of the edges adjacent to $w$ and $N(w)\cap[a_j,a_{j+1}-1]$.

Let $W = \{i \in [n-1] : v_i \in N(w) \}$. In the next claim, we show that $|W \cap [a_j, a_{j}+t_j] | \le 2$ for every $j \in [q]$. 
This is a fairly easy consequence of Claim~\ref{claim:pattern}.

\begin{claim} \label{clm:2}
Suppose that $|W \cap [a_j, a_{j}+t_j] | \ge 2$, then $W \cap [a_j, a_{j}+t_j]= \{r_j,s_j\}$ with $r_j < s_j$ and $wv_{r_j}, v_{s_j}w \in E(G)$.
Moreover, if $r_j \ne a_j$, then $v_{r_j} v_{r_j-1} \dots v_{a_j}$ is a directed path, and if $s_j \ne a_{j}+t_j$, then $v_{a_{j}+t_j} v_{a_{j}+t_j-1} \dots v_{s_j}$ is a directed path.
\end{claim}

\begin{proofclaim}
We first choose $r_j,s_j$ as follows. 
If $\sigma^-(J) = \ell-1$ and $i^* \in [a_j, a_{j}+t_j-1] \subseteq J$, then we set $r_j = i^*$ and $s_j=i^*+1$. 
By~\eqref{eqn:i^*}, we have $wv_{r_j}, v_{s_j}w \in E(G)$.
If $\sigma^-(J) = \ell$ or $i^* \notin [a_j, a_{j}+t_j-1]$, then pick $r_j, s_j \in W \cap [a_j, a_{j}+t_j]$ with $r_j < s_j$ and $[r_j,s_j] \cap W = \{r_j,s_j\}$. 
Note that $v_{s_j} v_{s_j-1} \dots v_{r_j}$ is a direct path, so $\sigma^-(v_{r_j} v_{r_j+1} \dots v_{s_j}) = e(v_{r_j} v_{r_j+1} \dots v_{s_j})$.
By Claim~\ref{claim:pattern}\ref{itm:pattern2} (with $i=r_j$ and $j=s_j$), we have $w v_{r_j}, v_{s_j} w  \in E(G)$.

In both cases, we have $\sigma^-( v_{a_j} \dots v_{r_j} ) = e ( v_{a_j} \dots v_{r_j} )$ and $\sigma^-( v_{s_j} \dots v_{a_{j}+t_j} ) = e ( v_{s_j} \dots v_{a_{j}+t_j} )$ implying the moreover statement. 

To complete the proof of the claim, it suffices to show that $W \cap [a_j,r_j] = \{r_j\}$ and $W \cap [s_j,a_{j}+t_j] = \{s_j\}$.
If there exists $i \in W \cap [a_j,r_j]$ and $i\ne r_j$, then Claim~\ref{claim:pattern}\ref{itm:pattern2} (with $j = r_j$) implies that $v_{r_j} w \in E(G)$, a contradiction. 
Hence $W \cap [a_j,r_j] = \{r_j\}$ and similarly $W \cap [s_j,a_{j}+t_j] = \{s_j\}$.
\end{proofclaim}

Note that Claim~\ref{clm:2} holds independent of $\sigma(C) = \ell$ or $\sigma(C) = \ell-1$. 

Let $a_{q+1} = n$. 
Our next goal is to establish an analogue result for the intervals $[a_j,a_{j+1}-1]$, subject to certain conditions.  
To this end, we define the following new parameter. For each $j \in [q]$, let 
\begin{align*}
	m_j := | [a_j, a_{j+1}-1] \setminus W | -t_j+1.
\end{align*}
Note that Claim~\ref{clm:2} implies that $| [a_j, a_{j}+t_j] \setminus W | \ge (t_j +1)-2 = t_j-1$.
Roughly speaking, $m_j$ counts the additional vertices in $\{v_i \in [a_{j}+t_j+1,a_{j+1}-1]\}$ that are not neighbours of~$w$. 
Note that 
\begin{align}
	|W \cap [a_j, a_{j}+t_j] |  \ge t_j+1 - | [a_j, a_{j+1}-1]  \setminus W | = 2 -m_j .
	\label{eqn:m_j}
\end{align}
We now show that $m_j \ge 0$ for every $j\in[q]$. Furthermore, if equality holds, then we gain some additional information of the orientation of the edges between $w$ and $\{v_i : i \in [a_j, a_{j+1}-1] \}$.

\begin{claim} \label{clm:m_j}
Suppose that $m_j \le 0$.
Then $m_j = 0$ and there exist $r^*_j,s^*_j \in W \cap [a_j, a_{j+1}-1]$ with $r^*_j < s^*_j$ such that
\begin{enumerate} [label={\rm(\roman*$'$)}]
	\item $wv_{r^*_j} \in E(G)$ and if $r^*_j \ne a_j$, then $v_{r^*_j} v_{r^*_j-1} \dots v_{a_j}$ is a directed path; \label{itm:m_j4}
	\item $v_{s^*_j} w \in E(G)$ and if $s^*_j \ne a_{j+1}-1$, then $v_{a_{j+1}-1} v_{a_{j+1}-2} \dots v_{s^*_j}$ is a directed path. \label{itm:m_j5}
\end{enumerate}
\end{claim}

\begin{proofclaim}
By~\eqref{eqn:m_j}, we have $|W \cap [a_j, a_{j}+t_j] | \ge 2-m_j \ge 2$. 
By Claim~\ref{clm:2}, we deduce that $|W \cap [a_j, a_{j}+t_j] |= \{r_j,s_j\}$ with $r_j < s_j$ and $wv_{r_j}, v_{s_j}w \in E(G)$.
This already implies that $m_j =0$. 
Set $r_j^* = r_j$ and so \ref{itm:m_j4} holds together with the moreover statement of Claim~\ref{clm:2}.

If $a_j+t_j=a_{j+1}-1$ then set $s_j^*= s_j$ and so \ref{itm:m_j5} holds (namely, $v_{a_{j+1}-1} v_{a_{j+1}-2} \dots v_{s^*_j}$ is a directed path). 
Suppose instead that $a_j+t_j<a_{j+1}-1$.
Since equality holds in~\eqref{eqn:m_j}, we have $[a_j+t_j+1, a_{j+1}-1] \subseteq W$. 
Note that $v_{a_j+t_j}v_{a_j+t_j+1} \in E(G)$ and, by the moreover statement of Claim~\ref{clm:2}, $v_{a_{j}+t_j} v_{a_{j}+t_j-1} \dots v_{s_j}$ is a direct path.
By Claim~\ref{claim:pattern}\ref{itm:pattern3} (with $i = s_j$ and $j = a_j+t_j+1$), we have $v_{a_j+t_j+1} w  \in E(G)$.
For each $i \in [a_j+t_j+1, a_{j+1}-2]$, given $v_i w  \in E(G)$ we deduce $v_{i+1}w\in E(G)$ by Claim~\ref{claim:pattern}\ref{itm:pattern3} (with $j =i+1$).\footnote{Or by considering the Hamilton cycle~$C'$ obtained from~$C$ by inserting $w$ in between $v_i$ and $v_{i+1}$, that is, $C' = v_1 \dots v_i w v_{i+1} \dots v_{n-1}1$. Note that $\sigma(C') \le \sigma(C)+ \sigma^-(v_i w v_{i+1})$.} In particular $v_{a_{j+1}-1}w\in E(G)$, so set $s^*_j = a_{j+1}-1$ and \ref{itm:m_j5} holds. 
\end{proofclaim}

\begin{claim}\label{clm:average}
Without loss of generality, we may assume that $m_1 = 0$ and that there exists $q^* \in [1,q]$  such that $m_{q^*+1} = 0$ and $m_j =1 $ for all $2\leq j\leq q^*$ (if $q^*=q$, we let $m_{q+1} := m_1$). 
\end{claim}

\begin{proofclaim}
Given $A\subseteq[q]$, let $m(A):=\sum_{j\in A} m_j$. Note that 
\begin{align}
	m([q])=\sum_{j \in [q]} m_j & = |[n-1] \setminus W| -\sum_{j \in [q]} t_j +q 
	= n-1 - d(w) -\sum_{j \in [q]} t_j +q
	\nonumber \\
	& \le \ell-1 -\sum_{j \in [q]} t_j +q
	\overset{\mathclap{\text{\eqref{eqn:sumt_j}}}}{=} q-1. \label{eqn:summ_j}
\end{align}
By~Claim~\ref{clm:m_j}, $m_j \ge 0$ for all $j \in [q]$. 
By~\eqref{eqn:summ_j}, $m_j=0$ for some $j$. 
We may assume that $m_1=0$. Partition $[q]$ into intervals $I_1,I_2,\dots,I_k$ so that $m_j=0$ if and only if $j$ is the smallest element of some interval $I_i$. Observe that 
$$\sum_{i \in [k]} m(I_i)=m([q])\overset{\mathclap{\text{\eqref{eqn:summ_j}}}}{<}q.$$
Hence, by the pigeonhole principle, there exists some $I_i$ such that $m(I_i)<|I_i|$. Up to relabelling, we may assume that $i=1$ and $I_1=[q^*]$ for some $q^*\in[q]$. By definition, we have $m_1=0$, $m_{q^*+1}=0$ and $m_j\ge 1$ for every $2\le j\le q^*$. In particular, we must have $m_j=1$ for every $2\le j\le q^*$, as otherwise $m(I_1)\ge |I_1|$. This concludes the proof.
\end{proofclaim}

Since $m_1= 0$, Claim~\ref{clm:m_j} implies that there exists $s^*_1 \in [a_1, a_2-1]$ with $v_{s^*_1} w \in E(G)$ and if $s^*_1 \ne a_{2}-1$, then $v_{a_{2}-1} v_{a_{2}-2} \dots v_{s^*_1}$ is a directed path. 
By mimicking the proof of Claim~\ref{clm:m_j}, we show that, for $j \in [2,q^*]$, we can find $s^*_j \in [a_j, a_{j+1}-1]$ with similar property.

\begin{claim} \label{clm:q^*}
For all $j\in[q^*]$, there exists $s^*_j \in W \cap [a_j, a_{j+1}-1]$ such that $v_{s^*_j} w \in E(G)$.
Moreover, if $s^*_j \ne a_{j+1}-1$, then $v_{a_{j+1}-1} v_{a_{j+1}-2} \dots v_{s^*_j}$ is a directed path.
\end{claim}

\begin{proofclaim}
We proceed by induction on~$j$. 
Recall that $m_1=0$, so the case $j=1$ follows immediately by Claim~\ref{clm:m_j}\ref{itm:m_j5}. 
Let $j \in [q^*]\setminus\{1\}$. 
We have $m_j=1$, hence $| W \cap [a_j,a_j+t_j] | \ge 2 -m_j = 1$ by~\eqref{eqn:m_j}. 

Let $s'_j = \min ( W \cap [a_j,a_j+t_j])$. 
We can assume that $s^*_{j-1}$ exists by induction hypothesis. 
In particular, $v_{s^*_{j-1}} w \in E(G)$ and the path $P'=v_{s^*_{j-1}} v_{s^*_{j-1}+1} \dots v_{a_{j}-1} v_{a_j}  \dots v_{s'_{j}}$ satisfies $\sigma^-(P') = e(P') -1$ as $v_{a_{j}-1} v_{a_j}$ is the only forward edge in~$P'$. 
By Claim~\ref{claim:pattern}\ref{itm:pattern3} (with $i = s^*_{j-1}$ and $j = s'_j$), we have 
$v_{s'_j}w \in E(G)$.

If $| W \cap [a_j,a_j+t_j] | \ge 2$, then Claim~\ref{clm:2} implies that $r_j = s'_j$ and $wv_{s'_j} \in E(G)$, a contradiction.
Therefore, we may assume that $ | W \cap [a_j,a_j+t_j] | = 1$.
This also implies that $i^* \notin [a_j,a_j+t_j-1]$ or else $i^*+1 \in W$ implies that $|[a_j, a_{j}+t_j] \cap W| \ge 2$.  
Therefore $v_{a_j+t_j} v_{a_j+t_j-1} \dots v_{s'_{j}}$ is a directed path. 

If $a_j + t_j = a_{j+1} - 1$, then set $s_j^*= s'_j$ and the claim  holds. 
Suppose instead that $a_j + t_j < a_{j+1} - 1$.
Since $ | W \cap [a_j,a_j+t_j] | = 1$ and equality holds in~\eqref{eqn:m_j}, we have $[a_j+t_j+1, a_{j+1}-1] \subseteq W$. 
Note that $\sigma^-(v_{s'_{j}}v_{s'_{j}+1} \dots v_{a_j+t_j+1}) = e(v_{s'_{j}}v_{s'_{j}+1} \dots v_{a_j+t_j+1}) -1 $. 
Since $v_{s'_j}w\in E(G)$, Claim~\ref{claim:pattern}\ref{itm:pattern3} (with $i = s'_j$ and $j = a_j+t_j+1$) implies that $v_{a_j+t_j+1} w  \in E(G)$. 
For each $i \in [a_j+t_j+1, a_{j+1}-2]$ in turns, given $v_i w  \in E(G)$, we deduce $v_{i+1}w\in E(G)$ by Claim~\ref{claim:pattern}\ref{itm:pattern3} (with $j =i+1$)\footnote{Or by considering the Hamilton cycle~$C'$ obtained from~$C$ by inserting $w$ in between $v_i$ and $v_{i+1}$, that is, $C' = v_1 \dots v_i w v_{i+1} \dots v_{n-1}1$. Note that $\sigma(C') \le \sigma(C)+ \sigma^-(v_i w v_{i+1})$.}.
In particular, $v_{a_{j+1}-1}w \in E(G)$, so set $s^*_j = a_{j+1}-1$ and the claim  holds. 
\end{proofclaim}
 
Recall that $m_{q^*+1} = 0$. 
Let $r^*_{q^*+1}$ and $s^*_{q^*}$ be obtained from Claim~\ref{clm:m_j}\ref{itm:m_j4} and  Claim~\ref{clm:q^*}, respectively. 
Note that we have $w v_{r^*_{q^*+1}},v_{s^*_{q^*}}w  \in E(G)$ and the path $P = v_{s^*_{q^*}} v_{s^*_{q^*}+1} \dots v_{r^*_{q^*+1}}$ satisfies $\sigma^-(P) = e(P) -1$.
This contradicts Claim~\ref{claim:pattern}\ref{itm:pattern3} (with $i = s^*_{q^*}$ and $j = r^*_{q^*+1})$.
\end{proof}

\section{Concluding remarks}\label{sec:conclusion}
In this paper, we consider the oriented discrepancy problem for Hamilton cycles. A natural research direction is to study oriented discrepancy problems for different conditions and structures. For example, it would be interesting to prove an analogue of Theorem~\ref{thm:main} with an Ore-type condition. 
It seems reasonable to consider other structures that have a natural notion of direction, such as powers of Hamilton cycles.

Given a graph $H$, one can assign an arbitrary orientation to~$H$. We then say that a copy of~$H$ in~$G$ has {\it large oriented discrepancy} if significantly more than half of its edges agree (or disagree) with the initial orientation. It would be interesting to determine which edge-orientation of $H$ minimises (or maximises) the minimum degree threshold required to ensure a host graph~$G$ contains a copy of~$H$ with a certain amount of discrepancy.

\section*{Acknowledgments}
The authors are grateful to Andrew Treglown and Sim\'on Piga for the fruitful discussions and comments and thank the referees for their detailed and helpful remarks.

\end{document}